\documentclass[a4paper,reqno]{amsart}
\usepackage{graphicx}

\usepackage[english]{babel}

\usepackage{amssymb,amsmath,mathtools}
\usepackage{amsthm}

\usepackage[left=1in,right=1in,top=1in,bottom=1in]{geometry}
\usepackage{hyperref}
\usepackage{color}

\newtheorem{defn}{Definition}
\newtheorem{theorem}{Theorem}

\newtheorem{remark}[theorem]{Remark}

\newtheorem{example}[theorem]{Example}

\theoremstyle{definition}
\newtheorem{as}{Assumption}[section]
\newcommand{\D}{\mathrm{d}}

\newcommand{\vx}{\textbf{\textit{x}}}

\newcommand{\ve}{\textbf{\textit{e}}}

\newcommand{\vv}{\textbf{\textit{v}}}

\newcommand{\vuu}{\textbf{\textit{u}}}

\newcommand{\supp}{\mathrm{supp}}
\newcommand{\sphere}{\mathbb{S}}
\mathtoolsset{showonlyrefs}

\usepackage{caption} 
\title{The head-wave ray transform}

\author[de Hoop]{Maarten V. de Hoop}
\address{Simons Chair in Computational and Applied Mathematics and Earth Science, Rice University, Houston, TX, USA (\tt{mdehoop@rice.edu})}

\author[Kykk\"anen]{Antti Kykk\"anen}
\address{Department of Computational Applied Mathematics and Operations Research, Rice University, Houston, TX, USA (\tt{ak272@rice.edu})}

\author[Mishra]{Rohit Kumar Mishra}
\address{Department of Mathematics, Indian Institute of Technology, Gandhinagar, Gujarat, India, (\tt{rohit.m@iitgn.ac.in})}

\date{\today}

\newcommand{\R}{{\mathbb R}}
\newcommand{\N}{{\mathbb N}}
\newcommand{\abs}[1]{\left\lvert#1\right\rvert}

\newcommand{\diver}{\mathrm{div}}

\begin{document}

\begin{abstract}
We introduce and study a new integral ray transform called the head-wave transform. The head-wave transform integrates a function along a piecewise linear (in general geodesic) path consisting of three parts. The geometry of such paths corresponds to ray paths of head-waves propagating in a medium with sharp changes in sound speed. The middle part of the ray paths corresponds to gliding along the so-called gliding surface. As our main results, we prove inversion formulas and null space characterizations under multiple different sets of assumptions on the geometry of the gliding surface and the integrand function.
\end{abstract}

\maketitle

\section{Introduction}
\label{sec:introduction}

We introduce and study a new kind of integral ray transform, which we call the head-wave transform. The head-wave transform integrates a function along a piecewise linear (in general geodesic) path consisting of three parts. Such ray transforms are connected to several types of ray transforms known in the literature, but in general cannot be reduced to any of the well-understood ones (See Section~\ref{sec:connections-to-other-transforms}).

The ray paths are connected to propagation of head-waves as described in \cite{CR1971}. The particular shape of the ray path arises from a sharp change in the sound speed of the medium (See Section~\ref{sec:head-waves-ala-cerveny}). If a propagating wave hits what we call a gliding surface at a critical angle, then the wave starts gliding along the surface and is later refracted back off it. The angles between the parts of the ray path and the gliding surface can vary and depend on the point where the ray hits the surface. In this article, we restrict attention to a single gliding surface in $\R^n$, which can either be flat or curved.

We prove several inversion formulas and null space characterizations for the head-wave transform under different sets of assumptions. Our proofs are mostly based on elementary techniques, but the geometry of head-waves and in particular, the variability of the angles between parts of the rays are completely novel in the literature. This article can be seen as an initial foray into the study of integral geometry and inverse problems related to the geometry of head-wave propagation.

\subsection{The head-wave transform}
\label{sec:definition}

In this section, we define the head-wave transform in $\R^n$ in the case when the gliding surface is flat dimension $n-1$ subspace, which we take to be $\R^{n - 1} \times \{0\} \approx \R^{n-1}$. Points in $\R^n$ will be denoted by $\vx = (\vx',x_n)$ where $\vx' \in \R^{n-1}$ and $x_n \in \R$. We will later on generalize the definition to account for cases when the gliding surface (or a codimension $1$ submanifold) is curved (See Section~\ref{sec:gliding-on-general-surface}).

As mentioned earlier, our rays of interest consist of three parts. The first and the last parts are straight lines in $\R^n$ and their directions are described by two smooth vector fields $\vuu, \vv \colon \R^{n-1} \times \sphere^{n-2} \to \sphere^{n-1}$. The vectors $\vuu$ and $\vv$ are assumed to have unit length. The middle part of the ray is a line segment in the gliding surface in $\R^n$.
We will have three parameters parametrizing our rays. Namely, a starting point $\vx' \in \R^{n-1}$, a gliding parameter $d \in \R_+$ and a unit vector $\theta \in \sphere^{n-2}$ in $\R^{n-1}$.
The gliding parameter $d$ signifies the length of the glide on the surface, and $\theta$ signifies the angle of the refracted ray.

We impose an additional requirement that the entire ray path lies in the plane spanned by the vectors $\vuu(\vx',\theta)$ and $\vv(\vx'+d\theta,\theta)$ for any fixed parameters $\vx'$, $d$ and $\theta$. This is equivalent to requiring that
\begin{equation}
\label{eqn:v-u-prime}
\vuu'(\vx',\theta)
=
\lambda_\vuu(\vx')\theta
\quad\text{and}\quad
\vv'(\vx',\theta)
=
\lambda_\vv(\vx')\theta
\end{equation}
for all $(\vx',\theta) \in \R^{n-1} \times \sphere^{n-2}$ where $\lambda_\vuu$ and $\lambda_\vv$ are smooth functions on $\R^{n-1}$. Additionally, assume that
\begin{equation}
\lambda_\vuu(\vx') < 0,
\quad
u_n(\vx',\theta) > 0,
\quad
\lambda_\vv(\vx') > 0
\quad\text{and}\quad
v_n(\vx',\theta)> 0
\end{equation}
for all $\vx' \in \R^{n-1}$ and $\theta \in \sphere^{n-2}$. Due to these assumptions the entire ray lies in the half-space $\R^n_+ = \R^{n-1} \times \R_+$ where $\R_+ = \{x \in \R\,:\,x \geq 0\}$. Thus, for the head-wave transform, it makes sense to restrict the support of the functions to this half-space.

With this setup, we are ready to define the integral transform of our interest.

\begin{defn}\label{def:refraction transform}
Let $\vuu, \vv \colon \R^{n-1} \times \sphere^{n-2} \to \R^n$ be smooth vector fields as above and let $f \colon \R^n_+ \to \R$ be a smooth function.
We define the head-wave transform $\mathcal{R}_{\vuu,\vv}f$ of $f$ as the map $\mathcal{R}_{\vuu,\vv}f:\R^{n-1} \times \R_+ \times \sphere^{n-2} \to \R$ defined by 
\begin{equation}
\label{eq:refraction transform}
\mathcal{R}_{\vuu,\vv}f(\vx',d,\theta)
\coloneqq
\int_0^\infty
f((\vx',0) + t\vuu(\vx',\theta))
\,\D t
+
\int_0^d
f(\vx'+t\theta,0)
\,\D t
+
\int_0^\infty
f((\vx'+d\theta,0)+t\vv(\vx'+d\theta,\theta))
\,\D t
\end{equation}
for all $\vx' \in \R^{n-1}$, $d \in \R_+$ and $\theta \in \sphere^{n-2}$.
\end{defn}

When no confusion can arise, we omit the reference to the vector fields $\vuu$ and $\vv$ and simply denote the head-wave transform by $\mathcal{R}$.
Using the assumptions made on the vector fields $\vuu$ and $\vv$ we can write the head-wave transform as
\begin{equation}
\begin{split}
\mathcal{R}f(\vx',d,\theta)
&=
\int_0^\infty
f(\vx'+t\lambda_\vuu(\vx')\theta,tu_n(\vx',\theta))
\,\D t
+
\int_0^d
f(\vx'+t\theta,0)
\,\D t
\\
&\quad+
\int_0^\infty
f(\vx'+(d+t\lambda_\vv(\vx'+d\theta))\theta,tv_n(\vx'+d\theta,\theta))
\,\D t.
\end{split}
\end{equation}

The head-wave transform integrates a scalar function over a piecewise linear (geodesic) trajectory given in Figure~\ref{fig:curve of integration}.

\begin{figure}[hbt]
    \centering
    \includegraphics[width=1\linewidth]{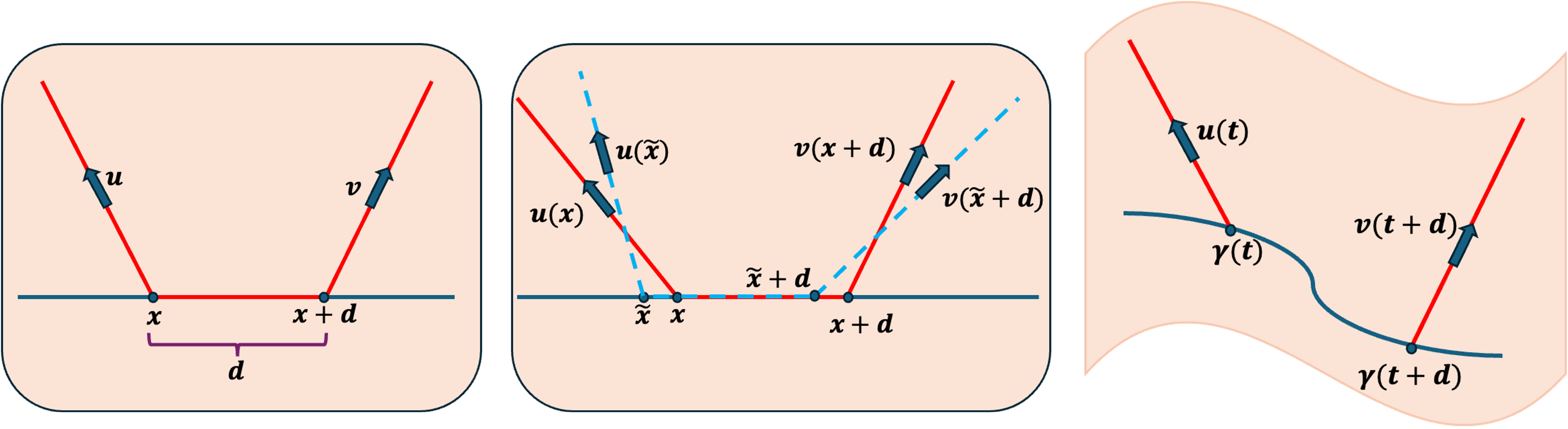}
    \caption{The setups of our theorems are built from ground up. We start from a simpler setup and increase the complexity of the ray paths step-by-step. First, we study the simplest case where the vectors $\vuu$ and $\vv$ are constant (left). Then, in the second step, we allow these vector fields to be variable (middle) and finally, in the third step, we allow the gliding surface (curve) to take more general shapes (right).}
    \label{fig:curve of integration}
\end{figure}

\subsection{Main results}
\label{sec:main-thms}

Our primary interest is to recover an unknown function $f$ from its head-wave transform. In dimensions $n \geq 3$, our problem is formally overdetermined as we are trying to find an $n$-dimensional object $f$ from a $(2n-2)$-dimensional data. Despite overdeterminacy, we expect our transform not to be injective in general, since our integral data contains information in two prescribed directions. In this article, we study two types of questions, which we discuss here briefly:
\begin{enumerate}
    \item[(i)] \label{item:question-1} \textbf{Inversion formula}: Under what assumptions on the unknown function $f$, can it be uniquely and explicitly recovered from the knowledge of its head-wave transform $\mathcal{R}f$?
    \vspace{2mm}
    \item[(ii)] \textbf{Null space description}: In general, the head-wave transform $\mathcal{R}$ is not injective and therefore it is natural to identify the obstruction to the injectivity. In other words, we would like to identify the null space of the head-wave transform $\mathcal{R}$.
\end{enumerate}
We provide answers to both the questions under several different sets of assumptions on the geometry of the vector fields $\vuu$ and $\vv$ and on the function $f$. First, we record various geometric assumptions on vector fields $\vuu$ and $\vv$ that will be used at different places in this article. These assumptions are listed in Assumption~\ref{as:A} and~Assumption~\ref{as:B} as follows:
\begin{as}
\label{as:A}
These assumptions will be needed to derive inversion formulas to recover a function. The first two conditions, (A1) and (A2), are related to the case when the gliding surface is $\R^{n-1} \times \{0\} \subseteq \R^n$. The third condition, (A3), is required in the case when gliding happens along a smooth curve $\gamma \colon \R \to \R^2$.

Notice that in (A1), since $n = 2$, the possible choices for the vector $\theta$ as defined earlier in this section are $\theta = 1$ and $\theta = -1$. Thus this assumption restricts set up even more. The added assumption is meaningful since in the physical setting described in Section~\ref{sec:head-waves-ala-cerveny} there would be waves coming in from the right (corresponding to $\theta = 1$) and then starting to propagate back to the right.

The first condition (A1) is satisfied if both of the vectors $\vuu$ and $\vv$ do not point along the gliding surface and neither are they normal to it, and the derivatives are non-zero and have the same sign. The second condition (A2) can be interpreted in a similar way and the third condition (A3) is the generalization of this in the case that the gliding surface is no longer flat.
\begin{enumerate}
    \item[(A1)] For $n = 2$, the vector fields $\vuu$ and $\vv$ satisfy
    \begin{equation}
    u_1(x) < 0,
    \quad
    u_2(x) > 0,
    \quad
    v_1(x) > 0,
    \quad
    v_2(x) > 0
    \end{equation}
    and
    \begin{equation}
    u'_1(x)v_1(x)(1-v_1(x))
    +
    v'_1(x)u_1(x)(1-u_1(x))
    \neq
    0
    \end{equation}
    for all $x \in \R$.
\vspace{2mm}
    \item[(A2)]
    Let $n \geq 3$. The vector fields $\vuu = (\vuu',u_n)$ and $\vv = (\vv',v_n)$ satisfy $\vuu'(\vx',\theta) = \lambda_\vuu(\vx')\theta$ and $\vv'(\vx',\theta) = \lambda_\vv(\vx')\theta$ for all $\vx'\in \R^{n-1}$ and $\theta\in \sphere^{n-2}$ where $\lambda_\vuu$ and $\lambda_\vv$ are smooth functions on $\R^{n-1}$. It holds that
    \begin{equation}
    \lambda_\vuu(\vx') < 0,
    \quad
    u_n(\vx',\theta) > 0,
    \quad
    \lambda_\vv(\vx') > 0
    \quad\text{and}\quad
    v_n(\vx',\theta) > 0
    \end{equation}
    and for a fixed $\theta_0 \in \sphere^{n-2}$ it holds that
    \begin{equation}
    (1 - \lambda_{\vv}(\vx'))
    \lambda_{\vv}(\vx')
    (\theta_0 \cdot \nabla_{\vx'}\lambda_{\vuu}(\vx'))
    +
    (1 - \lambda_{\vuu}(\vx'))
    \lambda_{\vuu}(\vx')
    (\theta_0 \cdot \nabla_{\vx'}\lambda_{\vv}(\vx'))
    \neq
    0
    \end{equation}
    for all $\vx' \in \R^{n-1}$.
\vspace{2mm}
    \item[(A3)] The vector fields $\vuu$ and $\vv$ along $\gamma$ satisfy
    \begin{equation}
    \vuu(t) \cdot \dot\gamma(t) < 0,
    \quad
    \vuu(t) \cdot \dot\gamma^\perp(t) > 0,
    \quad
    \vv(t) \cdot \dot\gamma(t) > 0
    \quad\text{and}\quad
    \vv(t) \cdot \dot\gamma^\perp(t) > 0
    \end{equation}
    and
    \begin{equation}
    -
    \left(
    \frac{u_1'(t_0)}{u_1(t_0)^2}
    +
    \frac{v_1'(t_0)}{v_1(t_0)^2}
    \right)
    \left(
    1
    -
    \frac{\gamma_1'(t_0)}{v_1(t_0)}
    \right)
    -
    \frac{v_1'(t_0)}{v_1(t_0)^2}
    \left(
    \frac{1}{u_1(t_0)}
    +
    \frac{1}{v_1(t_0)}
    \right)
    \neq
    0
    \end{equation}
    for all $t \in \R$.
\end{enumerate}

\end{as}
\begin{as}
\label{as:B}
To describe the null space of the head-wave transform, we will require the following different set of assumptions. The first two conditions, (B1) and (B2), are needed for the $2$-dimensional case, whereas the third condition, (B3), is related to higher-dimensional cases.
\vspace{2mm}
\begin{enumerate}
    \item[(B1)] The vector fields $\vuu$ and $\vv$ are extendable to $\R^2$ so that the integral curves of the extensions are straight lines.
\vspace{2mm}
    \item[(B2)] The vector fields $\vuu$ and $\vv$ are always linearly independent.
\vspace{2mm}
    \item[(B3)] For $n \geq 3$, the vector fields $\vuu$ and $\vv$ are independent of the variable $\vx' \in \R^{n-1}$ i.e. $\vuu$ and $\vv$ are functions of the type $\sphere^{n-1} \to \R^n$. 
\end{enumerate}
\end{as}
\noindent With these assumptions, we are now ready to record our main results with corresponding geometric assumptions in a Table~\ref{tab:all-results} below. 
\begin{table}[htb]
    \centering
    \begin{tabular}{c|c|c|c|}
    \cline{2-4}
     & Dimension & Assumptions on $\vuu$ and $\vv$ & Result \\ \hline\hline
     \multicolumn{1}{|c||}{Inversion formula} & \multicolumn{1}{c|}{$n = 2$} & (A1) & Theorem~\ref{thm:inversion-formula-2d} \\ 
     \multicolumn{1}{|c||}{ } & $n = 2$ & (A3) & Theorem~\ref{thm:inversion-smooth-curve} \\
     \multicolumn{1}{|c||}{ } & $n \geq 3$ & (A2) & Theorem~\ref{thm:inversion-formula-fixed-angle}
     \\ \hline
    \multicolumn{1}{|c||}{Null space description} & $n = 2$ & (B1) and (B2) & Theorem~\ref{thm:gauge-2d-var-angles} \\
     \multicolumn{1}{|c||}{ } & $n = 2$ & (B1) and (B2) & Theorem~\ref{thm:gauge-curved-2d} \\
     \multicolumn{1}{|c||}{ } & $n \geq 3$ & (B3) & Theorem~\ref{th:Gauge invariance fixed theta}\\
     \hline
    \end{tabular}
\caption{Geometric assumptions on the vector fields $\vuu$ and $\vv$. The assumption (A1)-(A3) and (B1)-(B3) are listed. In addition to the geometric assumptions, our theorems require certain assumptions for the function $f$, which can be found in the individual statements of the theorems.}
\label{tab:all-results}
\end{table}

\noindent We prove our main results in three steps. First in Section~\ref{sec:refraction-plane}, we begin by restricting ourselves to $n=2$, and then move to the general case in Section~\ref{sec:general-dimensions}. Because of our geometric assumptions on the vector fields $\vuu$ and $\vv$, we will be able to use slice-by-slice reconstruction using the results from $n=2$ for some of the general case $n \geq 2$. In Section~\ref{sec:gliding-on-general-surface} we generalize the $2$-dimensional results to cases where gliding happens on a curve.

\renewcommand{\theas}{\Alph{as}}

\subsection{Tomography and head-waves}
\label{sec:head-waves-ala-cerveny}

Consider a layered media consisting of two layers $U$ (upper)  and $L$ (lower) separated by an interface. Each layer is equipped with a sound speed. Assume that the sound speed in $U$ is slower than the sounds speed in $L$. When a seismic wave traveling in the slower layer $U$ encounters the interface with the faster layer $L$, there is a critical angle of incidence. At this angle, the refracted wave propagates along the interface at the higher speed of the layer $L$. As the wave propagates along the interface, it continuously radiates energy back into the layer $U$, producing a wavefront that returns to the surface --- this is called the head-wave.

In the ray picture corresponding to travel time measurements, the ray paths of head waves are formed of three parts: two roughly linear paths when the wave is traveling $U$ and a curve along the interface. From the point of view of our ray geometric setting, it is reasonable to assume data along rays for any gliding parameter $d \in \R_+$, since the wave radiates energy back to the surface continuously. In practice, this we would assume the travel times to be sufficiently densely sampled.

Seismic imaging with head-waves is sometimes called refraction tomography. This imaging modality uses travel times of head-waves recorded at the surface of the planet to determine the material parameters under the surface. Our inverse problem for head-wave transform is the corresponding linear inverse problem. The geophysical literature on seismic refraction goes back to at least~\cite{Hagedoorn1954,Hagedoorn1959}.

Usually one assumes that the seismic refraction imaging is done to a shallow enough region below the surface to that the material parameters can be assumed constant in depth. Our reconstruction formulas employ the corresponding assumption i.e. the function to be recovered is constant in the distance to the gliding surface. However, we do not make such assumption for proving null space characterizations.

The paths of integration in the head-wave transform are the ray paths given by propagation of singularities solutions to the scalar wave equation. It has been shown in the flat case~\cite{deHoop1990} that when the incidence angle with the gliding surface is critical then singularities start to propagate along the gliding surface. The general, non-flat, case has been considered in~\cite{CR1971} showing that the same should be true also in this case. No rigorous prove is given however. A formula for the critical angle is given in~\cite[Chapter 3, Equation (3.12)]{CR1971} and is compatible with assumptions (A1)-(A3) we make for the vectors $\vuu$ and $\vv$ in this article.

In addition to propagation of singularities one should keep track of how the wave is polarized along the ray path. Computations in~\cite{CR1971} suggest that as the head-wave hits the gliding surface at the critical angle the polarization can change.

This article contains an introduction to head-wave geometry and to mathematical inverse problems arising from imaging with head-waves. There are many possible directions for generalizations and variations of the mathematical problems. One can consider travel time problems where the data consist of lengths of gliding head-wave paths. Given this data one could try to determine the topology and the geometry of the gliding surface or one could try to reconstruct the vector $\vuu$ and $\vv$ describing the critical angles. Also, one could consider more general integral ray transform problems to generalize the linear problem.

\subsection{Connections to other kinds of ray transforms}
\label{sec:connections-to-other-transforms}

For $d =0$ and $\vuu$, $\vv$ constants, the above head-wave transform reduces to the well-known V-line/broken-ray transform of scalar function $f$ with vertices restricted to the $x$-axis. The V-line transform naturally arises in the field of integral geometry with scattering. The inversion problem for the V-line transform over scalar fields has been studied by several research groups in various geometries, for instance, see \cite{Ambartsoumian2012, ALJ2019, Florescu-Markel-Schotland, GZ2014, Kats_Krylov-13, Terz-review-18} and references therein. Recently, various generalizations of V-line/broken-ray transforms have been studied over higher-order tensor fields in different geometric settings, see \cite{AJM2020, AMZ2024, BMV2025, Ilmavirta_Parternain, MPZ2025}. For a complete literature review on the subject, please refer to a recent book by Gaik Ambartsoumian \cite{amb-book}.

In the absence of one of the ray directions ($\vuu$ or $\vv$), the V-line transform can be thought of as the classical divergent beam transform, which integrates scalar functions or, more generally, symmetric $m$-tensor fields over a beam starting from a point in the space. Please refer \cite{HSSW1980, Venky_Divergent_Beam_2025, Kuchment_Terzioglu} for works related to divergent beam transforms over scalar functions and symmetric $m$-tensor fields, respectively. Another related internal transform is the X-ray transform that integrates a function/tensor fields over straight lines, please refer \cite{Denisjuk_Paper, FLU, Helgason2011, PSU2023,  Sharafutdinov1994, Sharafutdinov_2021}. In certain situations, the head-wave transform can be used to determine the X-ray transform on the gliding surface.

\subsection*{Acknowledgements}

MVdH was supported by the National Science Foundation under grant DMS-2407456, the Simons Foundation under the MATH + X program and the corporate members of the Geo-Mathematical Imaging Group at Rice University. We thank the anonymous referees for the valuable remarks and suggestions.

\section{Proofs of the main results}
\label{sec:refraction-plane}

We begin by proving results for the head-wave transform in the plane $\R^2_+$. These results will later be used to prove results in higher dimensions in Section~\ref{sec:general-dimensions}.

\subsection{Inversion formulas in the plane}

For $n = 2$, the vector fields $\vuu$ and $\vv$ are of the type $\R \to \sphere^{n-1} \setminus \{(1,0),(-1,0)\}$ which means that
\begin{equation}
\vuu(x)
=
(u_1(x),u_2(x))
\quad\text{and}\quad
\vv(x)
=
(v_1(x),v_2(x))
\end{equation}
with $\abs{u_1(x)} \neq 1$, $\abs{v_1(x)} \neq 1$ and
\begin{equation}
u_1^2(x) + u_2^2(x) = 1 = v_1^2(x) + v_2^2(x)
\end{equation}
for all $x \in \R$. We also assume that $u_1(x) < 0$ and $v_1(x) > 0$ for all $x \in \R$.

Let $f$ be a smooth function defined on $\mathbb{R}^2_+$ that decays sufficiently fast at infinity. Suppose that there is a smooth function $\tilde f \colon \R \to \R$ so that $f(x, y) = f(x, 0) = \Tilde{f}(x)$. In this case, the head-wave transform of such an $f$ with vertices on the $x$-axis takes the following simplified form:
\begin{align}\label{def:2D refraction transform}
    \mathcal{R}f(x, d)
    &=
    \int_0^\infty
    \tilde{f}\left(x + tu_1(x)\right)
    \D t
    +
    \int_0^d
    \tilde{f}\left(x + t\right)
    \D t
    +
    \int_0^\infty
    \tilde{f}\left(x+ d + t v_1(x)\right)
    \D t
    \nonumber
    \\
    &=
    -\frac{1}{u_1(x)}
    \int_{-\infty}^x
    \tilde f(t)
    \,\D t
    +
    \int_x^{x+d}
    \tilde f(t)
    \,\D t
    +
    \frac{1}{v_1(x+d)}
    \int_{x+d}^\infty
    \tilde f(t)
    \,\D t.
\end{align}

Our first theorem below provides an inversion formula for this transform.

\begin{theorem}[Inversion formula for $n = 2$]
\label{thm:inversion-formula-2d}
Assume that the vector fields $\vuu$ and $\vv$ are smooth and suppose they satisfy the condition
\begin{equation}
\label{eqn:inversion-formula-2d-assumption}
u'_1(x)v_1(x)(1-v_1(x)) + v'_1(x)u_1(x)(1-u_1(x)) \neq 0
\end{equation}
for all $x \in \R$.
Let $f \in C^\infty(\R^2_+)$ and suppose there exists a function $\tilde f \in C^\infty(\R) \cap L^1(\R)$ so that $f(x,y) = \tilde f(x)$ for all $(x,y) \in \R^2_+$.
Then $f$ can be uniquely and explicitly recovered from the knowledge of its head-wave transform $\mathcal{R}f$ and the integral $\int_{-\infty}^\infty \tilde f(t)\,\D t$. In fact, we have
\begin{equation}
\label{eqn:inversion-formula-2d}
f(x,y)
=
\tilde f(x)
=
\frac{
\alpha'(x)\left.\partial_d\mathcal{R}f(x,d)\right|_{d = 0}+\beta'(x)[\partial_x\mathcal{R}f(x,0)-\zeta(x)]
}{
\alpha'(x)\beta(x) - \beta'(x)\alpha(x)
}
\end{equation}
where
\begin{equation}
\alpha(x) = \frac{1}{u_1(x)} + \frac{1}{v_1(x)},
\quad
\beta(x) = 1-\frac{1}{v_1(x)}
\quad\text{and}\quad
\zeta(x)
=
\frac{u'_1(x)}{u_1(x)^2}
\int_{-\infty}^\infty
\tilde f(t)
\,\D t.
\end{equation}
\end{theorem}

\begin{proof}
Let us define the function $g \colon \R \to \R$ by
\begin{equation}
g(x)
\coloneqq
\int_x^\infty \tilde f(t)\,\D t.
\end{equation}
Then by directly differentiating the data with respect to $x$ and $d$ we find that
\begin{equation}
\label{eqn:der-of-2d-data-1}
\begin{split}
\partial_x\mathcal{R}f(x,d)
&=
\left(
1
-
\frac{1}{v_1(x+d)}
\right)
\tilde f(x+d)
+
\left(
-1
-
\frac{1}{u_1(x)}
\right)
\tilde f(x)
\\
&\quad-
\frac{u_1'(x)}{u_1(x)^2}
g(x)
-
\frac{v_1'(x+d)}{v_1(x+d)^2}
g(x+d)
+
\frac{u'_1(x)}{u_1(x)^2}
\int_{-\infty}^\infty
\tilde f(t)
\,\D t
\end{split}
\end{equation}
and
\begin{equation}
\label{eqn:der-of-2d-data-2}
\partial_d\mathcal{R}f(x,d)
=
\left(
1
-
\frac{1}{v_1(x+d)}
\right)
\tilde f(x+d)
-
\frac{v_1'(x+d)}{v_1(x+d)^2}
g(x+d).
\end{equation}
Thus evaluating at $d = 0$ we get
\begin{equation}
\begin{split}
\partial_x\mathcal{R}f(x,0)
&=
-
\left(
\frac{1}{u_1(x)}
+
\frac{1}{v_1(x)}
\right)
\tilde f(x)
-
\left(
\frac{u_1'(x)}{u_1(x)^2}
+
\frac{v_1'(x)}{v_1(x)^2}
\right)
g(x)
\\
&=
\zeta(x)
-
\alpha(x)\tilde f(x)
+
\alpha'(x)g(x)
\end{split}
\end{equation}
and
\begin{equation}
\begin{split}
\left.
\partial_d\mathcal{R}f(x,d)
\right|_{d = 0}
&=
\left(
1-
\frac{1}{v_1(x)}
\right)
\tilde f(x)
-
\frac{v_1'(x)}{v_1(x)^2}
g(x)
\\
&=
\beta(x)\tilde f(x)
-
\beta'(x)g(x).
\end{split}
\end{equation}
Hence we find the relation
\begin{equation}
\alpha'(x)\left.\partial_d\mathcal{R}f(x,d)\right|_{d=0}
+
\beta'(x)\partial_x[\mathcal{R}f(x,0)-\zeta(x)]
=
(\alpha'(x)\beta(x) - \beta'(x)\alpha(x))\tilde f(x).
\end{equation}
A direct computation reveals that
\begin{equation}
\alpha'(x)\beta(x) - \beta'(x)\alpha(x)
=
\frac{u'_1(x)v_1(x)(1-v_1(x)) + v'_1(x)u_1(x)(1-u_1(x))}{u_1(x)^2v_1(x)^2}.
\end{equation}
Thus we obtain formula~\eqref{eqn:inversion-formula-2d} due to the assumption in equation~\eqref{eqn:inversion-formula-2d-assumption} as claimed.
\end{proof}

\begin{remark}
\label{rmk:2d-constant-u-v}
Assume that the vector fields $\vuu(x)$ and $\vv(x)$ are constants i.e. there are some $\vuu_0,\vv_0 \in \R^2$ so that $\vuu(x) \equiv \vuu_0$ and $\vv(x) \equiv \vv_0$. Then the inversion formula in Theorem~\ref{thm:inversion-formula-2d} cannot be used since assumption~\eqref{eqn:inversion-formula-2d-assumption} does not hold. In this case, we have the following:
\begin{enumerate}
    \item[(i)] Simpler inversion formulas can be derived using~\eqref{eqn:der-of-2d-data-1} and~\eqref{eqn:der-of-2d-data-2}. In fact, one directly computes that
    \begin{equation}
    \label{eqn:directly-for-d-equal-zero-1}
    f(x,y)
    =
    \tilde f(x)
    =
    -\frac{u_1v_1}{u_1+v_1}
    \partial_x\mathcal{R}f(x,0)
    \end{equation}
    and
    \begin{equation}
    f(x,y)
    =
    \tilde f(x)
    =
    \frac{v_1}{1-v_1}
    \left.\partial_d\mathcal{R}f(x,d)\right|_{d=0}
    \end{equation}
    and 
    \begin{equation}
    \label{eqn:directly-for-d-equal-zero-2}
    f(x,y) 
    =
    \tilde f(x)
    =
    \frac{u_1}{1+u_1}
    \left\{
    \partial_d \mathcal{R}f(x,d)
    -
   \partial_x \mathcal{R}f(x,d)
    \right\}.
    \end{equation}

    \item[(ii)] The function $f$ is uniquely determined by partial data and without the extra knowledge of $\int_{-\infty}^\infty \tilde f(t)\,\D t$.
    In fact, knowledge of the $1$-dimensional data $\mathcal{R}f(x,d_0)$ for all $x \in \R$ but a fixed $d_0 \in \R_+$ suffices to determine $f$ uniquely.
    Assume that $f \in C^\infty_c(\R^2_+)$ and $\tilde f \in C^\infty_c(\R)$ and suppose that $\mathcal{R}f(x,d_0) = 0$ for all $x \in \R$ but a fixed $d_0 \in \R_+$. Then, for $d_0 = 0$, we get $f = 0$ directly from either~\eqref{eqn:directly-for-d-equal-zero-1} or~\eqref{eqn:directly-for-d-equal-zero-2}.
    If $d_0 > 0$, then equation~\eqref{eqn:der-of-2d-data-1} gives the relation
    \begin{equation}
    \left(1- \frac{1}{v_1}\right)
    \tilde f(x+d_0)
    -
    \left(1+ \frac{1}{u_1}\right)
    \tilde f(x)
    =
    \partial_x \mathcal{R}f(x,d_0)
    =
    0.
    \end{equation}
    Hence, we find that
    \begin{equation}
    \tilde f(x+d_0) = C\tilde f(x)
    \quad\text{where}\quad
    C
    =
    \frac{v_1}{u_1}\frac{u_1+1}{v_1-1}.
    \end{equation}
    Iterating this formula yields
    \begin{equation}
    \tilde f(x+md_0)
    =
    C^m\tilde f(x)
    \end{equation}
    for all $x \in \R$ and $m \in \mathbb{N}$. Since the support of $\tilde f$ is compact, for any given $x \in \R$ there is $m \in \N$ so that $x+md_0$ lies outside the support. Hence, we conclude that $f(x,y) = \tilde f(x) = 0$ for all $(x,y) \in \R^2_+$. 
\end{enumerate}
\end{remark}

\subsection{Null space descriptions in the plane}

It is natural to ask if the assumption $f(x, y) = f(x, 0)$ can be removed and the full $f$ can be recovered from $\mathcal{R}f$. Since the head-wave transform data $\mathcal{R}f$ only contains integrals of $f$ in two prescribed directions, one generally expects the transform not to be injective. An obstruction to injectivity can easily be described in the special case when the vector fields $\vuu$ and $\vv$ are constant i.e. there are $\vuu_0,\vv_0 \in \R^2$ so that $\vuu(x) \equiv \vuu_0$ and $\vv(x) \equiv \vv_0$.

Consider a smooth compactly supported function $f \in C^\infty_c(\R^2)$ and let $\varphi \in C^\infty_c(\R^2)$ be a non-zero function so that $\varphi(\vx',0) = 0$ for all $\vx' \in \R$. 
Then, if we define $g \coloneqq f + \nabla_{\vuu_0}\nabla_{\vv_0} \varphi$, where $\nabla_{\vuu_0} = \vuu_0\cdot \nabla$ and $\nabla_{\vv_0} = \vv_0\cdot \nabla$, we have found two distinct functions $f$ and $g$ with equal head-wave transforms, since it is easily verified that
\begin{equation}
\mathcal{R}f
=
\mathcal{R}g
+
\mathcal{R}(\nabla_{\vuu_0}\nabla_{\vv_0} \varphi)
=
\mathcal{R}g.
\end{equation}
Our next theorem (Theorem~\ref{thm:kernel-characterization-2d-constants}) shows that this is the only obstruction to uniqueness. Such invariance of data is sometimes called gauge invariance.

\begin{theorem}[Null space description for constant angles]
\label{thm:kernel-characterization-2d-constants}
    Assume that the vector fields $\vuu$ and $\vv$ are constant i.e. there are some unit vectors $\vuu_0, \vv_0 \in \R^2$ so that $\vuu(x) \equiv \vuu_0$ and $\vv(x) \equiv \vv_0$. In addition, assume that
    \begin{equation}
    u_1 < 0,
    \quad
    u_2 > 0,
    \quad
    v_1 > 0
    \quad\text{and}\quad
    v_2 > 0.
    \end{equation}
    Let $f \in C^\infty_c(\R^2_+)$ be a smooth and compactly supported function. Then, the following two statements are equivalent for the head-wave transform:
\begin{enumerate}
    \item[(i)] $\mathcal{R}f(x,d) = 0$ for all $(x,d) \in \R^2_+$.
    \item[(ii)] The function $f$ vanishes on $\R \times  \{0\}$ and there is a function $\varphi \in C^\infty(\R^2_+)$ vanishing on $\R \times \{0\}$ so that $f = \nabla_{\vuu_0}\nabla_{\vv_0}\varphi$.
\end{enumerate}
\end{theorem}
\begin{proof}
    It is easy to see that statement (ii) implies (i) simply by the fundamental theorem of calculus. Therefore, we focus on proving the other direction, that is, (i) $\Longrightarrow$ (ii). 
    
    In the following, we use the notation $\tilde{x}$ to represent vector $(x, 0)$ in $\mathbb{R}^2$ and $\vuu = (u_1, u_2)$ and $\vv = (v_1, v_2)$. Assume for all $(x,d) \in \R \times \R_+$, we have $\mathcal{R}f(x,d) = 0$, which is the same as saying, 
    \begin{align}\label{eq:Rf=0}
        \int_0^\infty f\left(\tilde{x} + t\vuu\right)\D t + \int_0^d f \left(x + t, 0\right)\D t + \int_0^\infty f\left(\widetilde{x + d} + t \vv\right)\D t =  0, \ \ \forall (x,d) \in \R \times \R_+.
    \end{align}
    Differentiating the above equation with respect to $d$, we obtain
     \begin{align*}
     f \left(x + d, 0\right) + \int_0^\infty \frac{\partial f}{\partial x}\left(\widetilde{x + d} + t \vv\right)\D t =  0, \ \ \forall (x,d) \in \R \times \R_+.
    \end{align*}
    Next, let us apply the directional derivative $\nabla_\vv$ to the above equation and simplify to get 
      \begin{align*}
    \left(v_1 - 1\right) \frac{\partial f}{\partial x}\left(x + d, 0\right)&=  0, \ \ \forall (x,d) \in \R \times \R_+.\\
    \Longrightarrow \qquad \qquad \qquad \qquad  \frac{\partial f}{\partial x}\left(x + d, 0\right)&=  0, 
    \ \ \forall (x,d) \in \R \times \R_+ \ (\mbox{since } v_1 \neq 1 )\qquad \qquad\\
     \Longrightarrow \qquad \ \ \qquad \qquad \qquad \qquad  f\left(x, 0\right)&=  0,\ \ \forall x \in \R.
    \end{align*}
    This proves that $\displaystyle \mathcal{R}f(x,d) = 0 \implies f\left(x, 0\right)=  0$. With this, the equation \eqref{eq:Rf=0} reduces to the following identity: 
    \begin{align*}
            \int_0^\infty f\left(\tilde{x} + t\vuu\right)\D t  + \int_0^\infty f\left(\widetilde{x + d} + t \vv\right)\D t =  0, \ \ \forall (x,d) \in \R \times \R_+.
    \end{align*}
Following a similar line of argument as discussed above, one actually obtains that the integrals appearing in the above relation are simultaneously zero, that is, we have 
    \begin{align}\label{eq:X_u f = 0}
            \int_0^\infty f\left(\widetilde{x} + t\vuu\right)\D t  = 0 \quad \mbox{ and } \quad  \int_0^\infty f\left(\widetilde{x} + t \vv\right)\D t =  0, \ \ \forall x \in \R .
    \end{align}
Next, we claim that the above identities are true if and only if there exists a $\varphi \in C^\infty(\R \times \R_+)$ with $\varphi$ vanishes on $\R \times  \{0\}$ such that  $f = \nabla_{\vuu}\nabla_{\vv}\varphi$. Again, it is easy to see that  $$ f(x, y) = \nabla_{\vuu}\nabla_{\vv}\varphi(x, y) \implies     \int_0^\infty f\left(\widetilde{x} + t\vuu\right)\D t  = 0 \quad \mbox{ and } \quad  \int_0^\infty f\left(\widetilde{x} + t \vv\right)\D t =  0, \ \ \forall x \in \R .$$
To prove the other direction, let us start with 
\begin{align*}
    \int_0^\infty f\left(\widetilde{x} + t\vuu\right)\D t  = 0, \ \ \forall x \in \R . 
\end{align*}
Define $$\psi(x,y) =  \int_0^{y/u_2} f\left((x-(u_1/u_2)y, 0) + t\vuu\right)\D t.$$ 
With this choice of $\psi$, one may verify by a direct calculation that $\displaystyle \psi(x,0) = 0$ and $\displaystyle f(x, y) =  \nabla_\vuu\psi(x, y)$.  Now, let us consider the second integral from equation \eqref{eq:X_u f = 0}
\begin{align*}
    \int_0^\infty f\left(\widetilde{x} + t \vv\right)\D t &=  0, \ \ \forall x \in \R \\
    \Longrightarrow    \quad\qquad\qquad\int_0^\infty \nabla_\vuu\psi\left(\widetilde{x} + t \vv\right)\D t &=  0, \ \ \forall x \in \R \qquad\qquad\qquad\\
    \Longrightarrow \  \qquad \qquad\qquad  \int_0^\infty \psi\left(\widetilde{x} + t \vv\right)\D t &=  0, \ \ \forall x \in \R. \qquad\qquad
\end{align*}
Finally, define  $$\varphi(x,y) =  \int_0^{y/v_2}\psi\left((x-(v_1/v_2)y, 0) + t\vv\right)\D t.$$ 
Then  $\displaystyle \varphi(x,0) = 0$ and $\displaystyle \psi(x, y) =  \nabla_\vv\varphi(x, y)$. Hence $\displaystyle f(x, y) =  \nabla_\vuu\psi(x, y) =  \nabla_\vuu \nabla_\vv\varphi(x, y)$ as claimed. This completes the proof of the theorem. 
\end{proof}

Before moving to higher dimensional head-wave transforms we prove a null space description result (Theorem~\ref{thm:gauge-2d-var-angles}) when the vector fields $\vuu$ and $\vv$ are allowed to be variable. To state the theorem, we define the variable coefficient partial differential operators
\begin{equation}
P_\vuu(x,y)
=
\vuu(x,y) \cdot \nabla
\quad\text{and}\quad
P_\vv(x,y)
=
\vv(x,y) \cdot \nabla.
\end{equation}
Let us denote
\begin{equation}
A_{\vuu,\vv}(x,y)
\coloneqq
\begin{pmatrix}
\vuu(x,y) & \vv(x,y)
\end{pmatrix}
=
\begin{pmatrix}
u_1(x,y) & v_1(x,y) \\
u_2(x,y) & v_2(x,y)
\end{pmatrix}.
\end{equation}
Theorem~\ref{thm:gauge-2d-var-angles} requires an additional technical assumption that the function $f \in C^\infty_c(\R^2)$ satisfies
\begin{equation}
\label{eqn:extra-assumption}
\diver(\vuu(x,y))
\int_0^\infty
f((x,y) + t\vuu(x,y))
\,\D t
-
\diver(\vv(x,y))
\int_0^\infty
f((x,y) + t\vv(x,y))
\,\D t
=
0
\end{equation}
for all $(x,y) \in \R^2_+$.

\begin{theorem}[General gauge description for $n = 2$]
\label{thm:gauge-2d-var-angles}
Assume that the vector fields $\vuu$ and $\vv$ are smooth and always linearly independent. In addition, assume that $\vuu$ and $\vv$ can be extended to $\R^2$ so that their integral curve are straight lines.
Let $f \in C^\infty_c(\R^2_+)$ be a function satisfying equation~\eqref{eqn:extra-assumption} and suppose that $\supp(f) \subseteq \{y > 0\}$. Then the following statements are equivalent:
\begin{enumerate}
    \item[(i)] $\mathcal{R}f(x,d) = 0$ for all $(x,d) \in \R^2_+$.

    \item[(ii)] There is a function $\varphi \in C^\infty_c(\R^2_+)$ vanishing on $\R \times \{0\}$ so that
    \begin{equation}
    f(x,y) 
    =
    P_\vuu(x,y)
    [\det(A_{\vuu,\vv}(x,y))^{-1}
    P_\vv(x,y)\varphi(x,y)]
    \end{equation}
    and also that
    \begin{equation}
    f(x,y) 
    =
    P_\vv(x,y)
    [\det(A_{\vuu,\vv}(x,y))^{-1}
    P_\vuu(x,y)\varphi(x,y)].
    \end{equation}
\end{enumerate}
\end{theorem}

\begin{remark}
Theorem~\ref{thm:gauge-2d-var-angles} makes assumptions on  vector field $\vuu$ and $\vv$ which look somewhat restricting, in particular equation~\eqref{eqn:extra-assumption}. In this remark we give three different conditions under which all of the assumptions on these vector fields are satisfied.
\begin{enumerate}
    \item[(i)] \textbf{Both $\vuu$ and $\vv$ are constants}: Let $\vuu(x,y) = (-1,1)$ and $\vv(x,y) = (1,1)$ for all $(x,y) \in \R^2_+$. Then $\vuu$ and $\vv$ are linearly independent,  all integral curves of both fields are straight lines and equation~\eqref{eqn:extra-assumption} is satisfied since $\diver(\vuu) = \diver(\vv) = 0$.

    \item[(ii)] \textbf{The field $\vuu$ is constant but $\vv$ is variable}: Let $\vuu(x,y) = (-1,1)$ and $\vv(x,y) = (1+(x-y)^2,1+(x-y)^2)$ for all $(x,y) \in \R^2_+$. Then $\vuu$ and $\vv$ are linearly independent, all integral curves of both fields are straight lines and equation~\eqref{eqn:extra-assumption} is satisfied since $\diver(\vuu) = \diver(\vv) = 0$.

    \item[(iii)] \textbf{Both $\vuu$ and $\vv$ are variable}: Let $\vuu(x,y) = (-1-(x-y)^2,1+(x-y)^2)$ and $\vv(x,y) = (1+(x-y)^2,1+(x-y)^2)$ for all $(x,y) \in \R^2_+$. Then $\vuu$ and $\vv$ are linearly independent and integral curves of both fields are straight lines. It holds that $\diver(\vuu) = -4(x-y)$ and $\diver(\vv) = 0$. In this case equation~\eqref{eqn:extra-assumption} is satisfied if we are given the a priori knowledge that
    \begin{equation}
    \int_0^\infty
    f((x,y) + t\vuu(x,y))
    \,\D t
    =
    0
    \end{equation}
    for all $(x,y) \in \R^2_+$.
\end{enumerate}
\end{remark}

\begin{proof}[Proof of Theorem~\ref{thm:gauge-2d-var-angles}]
Let us first assume that (ii) is true. Let us denote
\begin{equation}
\psi_1(x,y)
\coloneqq
\det(A_{\vuu,\vv}(x,y))
P_\vv(x,y)\varphi(x,y)
\end{equation}
and
\begin{equation}
\psi_2(x,y)
\coloneqq
\det(A_{\vuu,\vv}(x,y))
P_\vuu(x,y)\varphi(x,y).
\end{equation}
Then, since the integral curves of $\vuu(x,y)$ and $\vv(x,y)$ are straight lines, the head-wave transform of $f$ can be written as
\begin{equation}
\mathcal{R}f(x,d)
=
\int_0^\infty
P_\vuu(x,0)\psi_1((x,0) + t\vuu(x,0))
\,\D t
+
\int_0^\infty
P_\vv(x+d,0)\psi_2((x+d,0) + t\vv(x+d,0))
\,\D t
\end{equation}
This expression equals $0$ due to the fundamental theorem of calculus. Hence (i) holds.

Let us then assume that $\mathcal{R}f(x,d) = 0$ for all $(x,d) \in \R^2_+$. First, we construct auxiliary functions $\psi_\vuu,\psi_\vv \in C^\infty_c(\R^2_+)$ so that $\nabla_\vuu \psi_\vuu = f = \nabla_\vv \psi_\vv$ and $\psi_\vuu|_{y = 0} = 0 = \psi_\vv|_{y = 0}$. The vanishing of the head-wave transform is equivalent to
\begin{equation}
\int_0^\infty
f(x+tu_1(x,0),tu_2(x,0))
\,\D t
+
\int_0^d
f(x+t,0)
\,\D t
+
\int_0^\infty
f(x+d+tv_1(x+d,0),tv_2(x+d,0))
\,\D t
=
0.
\end{equation}
Let us then define the functions $\psi_\vuu, \psi_\vv \colon \R^2_+ \to \R$ by the formulas
\begin{equation}
\psi_\vuu(x,y)
\coloneqq
-
\int_0^\infty
f((x,y) + t\vuu(x,y))
\,\D t
\end{equation}
and
\begin{equation}
\psi_\vv(x,y)
\coloneqq
-
\int_0^\infty
f((x,y) + t\vv(x,y))
\,\D t.
\end{equation}
Since $f \in C^\infty_c(\R^2_+)$, we have $\psi_\vuu, \psi_\vv \in C^\infty_c(\R^2_+)$. Additionally, it follows directly from the fundamental theorem of calculus that $P_{\vuu}(x,y)\psi_\vuu(x,y) = f(x,y)$ and $P_{\vv}(x,y)\psi_\vv(x,y) = f(x,y)$. Since the support of $f$ lies in $\{y > 0\}$ the vanishing of the head-wave transform gives
\begin{equation}
0
=
-\mathcal{R}f(x,d)
=
\psi_\vuu(x,0)
+
\psi_\vv(x+d,0)
\end{equation}
for all $x \in \R$ and $d \in \R_+$. Since this holds true for all $d \in \R_+$ we have that $\psi_\vuu|_{y = 0} = 0$ and consequently, $\psi_\vv|_{y = 0} = 0$. Next, we use these auxiliary functions to construct the required function $\varphi$.

Consider the $1$-form $\omega$ defined by
\begin{equation}
\omega
\coloneqq
(v_2\psi_\vv - u_2\psi_\vuu)dx
+
(u_1\psi_\vuu - v_1\psi_\vv)dy.
\end{equation}
We denote its components by
\begin{equation}
\omega_1(x,y)
\coloneqq
v_2(x,y)\psi_\vv(x,y)
-
u_2(x,y)\psi_\vuu(x,y)
\quad\text{and}\quad
\omega_2(x,y)
\coloneqq
u_1(x,y)\psi_\vuu(x,y)
-
v_1(x,y)\psi_\vv(x,y).
\end{equation}
It holds that
\begin{equation}
\partial_1\omega_2
=
(\partial_1u_1)\psi_\vuu
+
u_1(\partial_1\psi_\vuu)
-
(\partial_1v_1)\psi_\vv
-
v_1(\partial_1\psi_\vv)
\end{equation}
and
\begin{equation}
\partial_2\omega_1
=
(\partial_2v_2)\psi_\vv
+
v_2(\partial_2\psi_\vv)
-
(\partial_2u_2)\psi_\vuu
-
u_2(\partial_2\psi_\vuu).
\end{equation}
Hence, the exterior derivative of $\omega$ is
\begin{equation}
\begin{split}
d\omega
&=
(\partial_1\omega_2 - \partial_2\omega_1)
dx \wedge dy
\\
&=
(\diver(\vuu)\psi_\vuu
-
\diver(\vv)\psi_\vv
+
P_\vuu \psi_\vuu
-
P_\vv \psi_\vv)
dx \wedge dy
\\
&=
(\diver(\vuu)\psi_\vuu
-
\diver(\vv)\psi_\vv)dx \wedge dy
\\
&=
0
\end{split}
\end{equation}
where the last line follows from the assumption~\eqref{eqn:extra-assumption} made on $f$. Since both $\psi_\vuu$ and $\psi_\vv$ are compactly supported and smooth, we have shown that $\omega$ is a closed compactly supported smooth $1$-form. Hence is follows from the Poincare lemma~\cite[Corollary 8.3.17]{Conlon2008} that $\omega$ is exact i.e. there is some $\varphi \in C^\infty_c(\R^2_+)$ so that $d\varphi = \omega$. In particular, $\partial_1\varphi = \omega_1$ and $\partial_2\varphi = \omega_2$. Thus
\begin{equation}
\begin{split}
P_{\vuu}(x,y) \varphi(x,y)
&=
u_1(x,y)\omega_1(x,y)
+
u_2(x,y)\omega_2(x,y)
\\
&=
u_1(x,y)v_2(x,y)\psi_\vv(x,y)
-
u_1(x,y)u_2(x,y)\psi_\vuu(x,y)
\\
&\quad
+
u_2(x,y)u_1(x,y)\psi_\vuu(x,y)
-
u_2(x,y)v_1(x,y)\psi_\vv(x,y)
\\
&=
\det(A_{\vuu,\vv}(x,y))\psi_\vv(x,y).
\end{split}
\end{equation}
Consequently, we find that
\begin{equation}
P_{\vv}(x,y)(\det(A_{\vuu,\vv}(x,y))^{-1}P_{\vuu}(x,y)\varphi(x,y))
=
P_{\vv}(x,y)\psi_\vv(x,y)
=
f(x,y).
\end{equation}
A similar computation reveals that also
\begin{equation}
P_{\vuu}(x,y)(\det(A_{\vuu,\vv}(x,y))^{-1}P_{\vv}(x,y) \varphi(x,y))
=
P_{\vuu}(x,y)\psi_\vuu(x,y)
=
f(x,y).
\end{equation}
Thus $\varphi$ is the function we set out to construct.
\end{proof}

\section{The head-wave transform in general dimensions}
\label{sec:general-dimensions}

As mentioned earlier, the inversion of the head-wave transform is formally overdetermined in $\mathbb{R}^n$ when $n \geq 3$. More specifically, the head-wave transform $\mathcal{R}f$ depends on $(2n-2)$ variables, whereas $f$ depends only on $n$ parameters. Therefore, it is natural to consider this inversion question with lower-dimensional data sets. Below, we have considered various physically relevant setups to reduce the dimensionality of the dataset and address the invertibility questions in each setup.

\subsection{Results for a fixed gliding direction}

Recall, $\mathcal{R}f$ depends on $\vx^\prime \in \R^{n-1}$, $d\in  \R_+$, and $\theta_0 \in \sphere^{n-2}$. In this section, we fix a $\theta_0 \in \sphere^{n-2}$ and consider $n$-dimesional data $\mathcal{R}f(\vx^\prime, d, \theta_0)$ for $\vx^\prime \in \R^{n-1}$, $d\in  \R_+$ to reconstruct the scalar function $f$ defined on $\R^{n}$. For a given $\vx = (\vx^\prime, x_n) \in \R^n$, the entire refracted line will lie in the plane $S_{\vx^\prime, \theta_0}$ (the plane that is perpendicular to $\{x_n = 0\}$ and containing the line $\{ \vx^\prime + t\theta_0: t \in \R\}$). This plane  $S_{\vx^\prime, \theta_0}$ can be parameterized by new coordinates $(x_1, y_1)$. Once we make these identifications, geometrically, this case ($\theta_0$ being fixed) is very similar to the plane case discussed in Section \ref{sec:refraction-plane} above. One can apply the results proven in Section \ref{sec:refraction-plane} slice-by-slice (on parallel slices to the plane $S_{\vx^\prime, \theta_0}$) and obtain analogous results in this case as well. 

Again, we assume $f$ is a function independent of the variable $x_n$, i.e., there is a smooth compactly supported function $\tilde f: \R^{n-1} \to \R$ such that $f(\vx) = \tilde f(\vx')$. Recall that $\vuu'(\vx',\theta_0) = \lambda_\vuu(\vx')\theta_0$ and $\vv'(\vx',\theta_0) = \lambda_\vv(\vx')\theta_0$ for some smooth functions $\lambda_\vuu, \lambda_\vv \colon \R^{n-1} \to \R$. In addition, we assume that $\lambda_\vuu(\vx') < 0$ and $\lambda_\vv(\vx') > 0$. Then the head-wave transform reduces to
\begin{equation}
\begin{split}
\mathcal{R}f(\vx',d,\theta_0)
&=
\int_0^\infty
\tilde f(\vx'+t\vuu'(\vx',\theta_0))
\,\D t
+
\int_0^d
\tilde f(\vx'+t\theta_0)
\,\D t
+
\int_0^\infty
\tilde f(\vx'+d\theta_0+t\vv'(\vx'+d\theta_0,\theta_0))
\,\D t
\\
&=
-\frac{1}{\lambda_\vuu(\vx')}
\int_{-\infty}^0
\tilde f(\vx'+t\theta_0)
\,\D t
+
\int_0^d \tilde f(\vx'+t\theta_0)
\,\D t
+
\frac{1}{\lambda_\vv(\vx'+d\theta_0)} \int_d^\infty
\tilde f(\vx'+t\theta_0)
\,\D t.
\end{split}
\end{equation}
In this reduced case, we have the inversion results similar to the $2$-dimensional case discussed in Theorem \ref{thm:inversion-formula-2d} of Section \ref{sec:refraction-plane}. We use the notation $D_{\theta_0} \coloneqq \theta_0 \cdot \nabla_{\vx'}$.

\begin{theorem}[Inversion formula for fixed $\theta$]
\label{thm:inversion-formula-fixed-angle}
Fix an angle $\theta_0 \in \sphere^{n-2}$. Assume that
\begin{equation}
(1 - \lambda_{\vv}(\vx'))
\lambda_{\vv}(\vx')
D_{\theta_0}\lambda_{\vuu}(\vx')
+
(1 - \lambda_{\vuu}(\vx'))
\lambda_{\vuu}(\vx')
D_{\theta_0}\lambda_{\vv}(\vx')
\neq
0
\end{equation}
for all $\vx' \in \R^{n-1}$.
Let $f \in C^\infty_c(\R^n_+)$ and suppose there is a function $\tilde f \in C^\infty_c(\R^{n-1})$ so that $f(\vx) = \tilde f(\vx')$ for all $\vx = (\vx',x_n) \in \R^n_+$. Then the function $f$ can be uniquely and explicitly recovered from the knowledge of its head-wave transform $\mathcal{R}f(\vx',d,\theta_0)$ for all $\vx' \in \R^{n-1}$ and $d \in \R_+$ and the integrals $\int_{-\infty}^\infty \tilde f(\vx' + t\theta_0)\,\D t$ for all $\vx' \in \R^{n-1}$. In fact, we have
\begin{equation}
f(\vx',x_n)
=
\tilde f(\vx')
=
\frac{
D_{\theta_0}\beta(\vx')D_{\theta_0}[\mathcal{R}f(\vx',0,\theta_0)-\zeta(\vx')]
+
D_{\theta_0}\alpha(\vx')\left.\partial_d\mathcal{R}f(\vx',d,\theta_0)\right|_{d=0}
}{
\beta(\vx')D_{\theta_0}\alpha(\vx')
-
\alpha(\vx')D_{\theta_0}\beta(\vx')
}
\end{equation}
where
\begin{equation}
\alpha(\vx')
=
\frac{1}{\lambda_{\vuu}(\vx')}
+
\frac{1}{\lambda_{\vv}(\vx')},
\quad
\beta(\vx')
=
1
-
\frac{1}{\lambda_{\vv}(\vx')}
\quad\text{and}\quad
\zeta(\vx')
=
\frac{D_{\theta_0}\lambda_{\vuu}(\vx')}{\lambda_{\vuu}(\vx')^2}
\int_{-\infty}^\infty
\tilde f(\vx' + t\theta_0)
\,\D t.
\end{equation}
\end{theorem}

\begin{proof}
Let us define the function $g \coloneqq \R^{n-1} \to \R$ by
\begin{equation}
g(\vx')
\coloneqq
\int_0^\infty
\tilde f(\vx' + t\theta_0)
\,\D t.
\end{equation}
Directly computing the derivatives $D_{\theta_0}$ and $\partial_d$ of the data we find the formulas
\begin{equation}
\label{eqn:reconstruction-nd-1}
D_{\theta_0}\mathcal{R}f(\vx',0,\theta_0)
=
\zeta(\vx')
-
\alpha(\vx')\tilde f(\vx')
+
g(\vx')D_{\theta_0}\alpha(\vx')
\end{equation}
and
\begin{equation}
\label{eqn:reconstruction-nd-2}
\left.
\partial_d
\mathcal{R}f(\vx',d,\theta_0)
\right|_{d=0}
=
\beta(\vx')\tilde f(\vx')
-
g(\vx')D_{\theta_0}\beta(\vx').
\end{equation}
The proof can be completed in perfect analogy with the proof of Theorem~\ref{thm:inversion-formula-2d}.
\end{proof}

\begin{remark}
Formulas~\eqref{eqn:reconstruction-nd-1} and~\eqref{eqn:reconstruction-nd-2} can be used to derive a connection between the head-wave transform of $f$ and the X-ray transform of $\tilde f$ on the gliding surface $\R^{n-1}$.

Assume that $D_{\theta_0}\beta(\vx' + s\theta_0)$ has a well-defined limit $C \in \R \setminus \{0\}$ as $s \to -\infty$. It follows from~\eqref{eqn:reconstruction-nd-2} that
\begin{equation}
\left.
\partial_d
\mathcal{R}f(\vx'+s\theta_0,d,\theta_0)
\right|_{d=0}
=
\beta(\vx'+s\theta_0)\tilde f(\vx'+s\theta_0)
-
D_{\theta_0}\beta(\vx'+s\theta_0)
\int_{s}^\infty \tilde f(\vx' + t\theta_0)\,\D t.
\end{equation}
Hence taking the limit $s \to -\infty$ gives
\begin{equation}
\lim_{s \to \-\infty}
\left.
\partial_d
\mathcal{R}f(\vx'+s\theta_0,d,\theta_0)
\right|_{d=0}
=
-C\mathcal{X}\tilde f(\vx',\theta_0)
\end{equation}
where $\mathcal{X}f(\vx',\theta_0)$ denotes the standard longitudinal X-ray transform of $\tilde f$ along the line through $\vx'$ into the direction of $\theta_0$ in $\R^{n-1}$. In conclusion, we have shown that the head-wave transform data can be used to determine the X-ray transform of $\tilde f$ into the directions allowed by the considered problem.
\end{remark}

Our gauge invariance result (Theorem~\ref{th:Gauge invariance fixed theta}) for fixed angle $\theta_0 \in \sphere^{n-2}$
assumes that the vector fields $\vuu$ and $\vv$ are independent of the variable $\vx' \in \R^{n-1}$. Moreover, $\vuu(\theta) = (\lambda_\vuu\theta,u_n(\theta))$ and $\vv(\theta) = (\lambda_\vv\theta,v_n(\theta))$ for some constants $\lambda_\vuu,\lambda_\vv \in \R$. In this case the head-wave transform takes the form
\begin{equation}
\mathcal{R}f(\vx,d,\theta_0)
=
\int_0^\infty
f(\vx' + t\lambda_\vuu\theta_0,tu_n(\theta_0))
\,\D t
+
\int_0^d
f(\vx' + t\theta_0,0)
\,\D t
+
\int_0^\infty
f(\vx'+d\theta_0+t\lambda_\vv\theta_0,tv_n(\theta_0))
\,\D t.
\end{equation}

We use a slice-by-slice approach to promote the $2$-dimensional Theorem~\ref{thm:kernel-characterization-2d-constants} to obtain the following.

\begin{theorem}[Gauge invariance for fixed $\theta$]
\label{th:Gauge invariance fixed theta}
Assume that the vector fields $\vuu$ and $\vv$ are smooth and independent of the variable $\vx' \in \R^{n-1}$ i.e. $\vuu$ and $\vv$ are functions of the type $\sphere^{n-2} \to \R^{n}$. In addition, assume that
\begin{equation}
\lambda_\vuu < 0,
\quad
u_n(\theta) > 0,
\quad
\lambda_\vv > 0
\quad\text{and}\quad
v_n(\theta) > 0.
\end{equation}
Let $f \in C^\infty_c(\R^n_+)$ be a smooth compactly supported function. Fix an angle $\theta_0 \in \sphere^{n-2}$. Then the following are equivalent:
\begin{enumerate}
    \item[(i)] $\mathcal{R}f(\vx',d,\theta_0) = 0$ for all $(\vx',d) \in \R^{n-1} \times \R_+$.
    \item[(ii)] The function $f$ vanishes on $\R^{n-1} \times \{0\}$ and $\displaystyle f(\vx) = \nabla_{\vuu(\theta_0)}\nabla_{\vv(\theta_0)}\varphi(\vx)$ for some $\varphi \in C^\infty_c(\R^n_+)$ that vanishes on $\R^{n-1} \times \{0\}$.
\end{enumerate}
\end{theorem}

\begin{proof}
Assuming (ii) to be true, item (i) is proved by an application of the fundamental theorem of calculus. Let us thus assume that (i) is true and we will prove (ii).

First, take the derivative $\partial_d$ of the equation $\partial_d \mathcal{R}f(\vx,d,\theta_0) = 0$ to obtain
\begin{equation}
\lambda_\vv D_{\theta_0}f(\vx'+d\theta_0,0)
-
D_{\theta_0}f(\vx'+d\theta_0,0) = 0.
\end{equation}
Since $\lambda_\vv \neq 1$, it follows that $D_{\theta_0}f(\vx',\theta_0) = 0$ for all $\vx' \in \R^{n-1}$ and hence $f$ vanishes on $\R^{n-1} \times \{0\}$.

Now, the head-wave transform data of $f$ gives the relation
\begin{equation}
\label{eqn:gauge-fixed-angle-proof-1}
\int_0^\infty
f(\vx'+t\lambda_\vuu\theta_0,tu_n(\theta_0))
\,\D t
+
\int_0^\infty
f(\vx'+d\theta_0+t\lambda_\vuu\theta_0,tu_n(\theta_0))
\,\D t
=
0
\end{equation}
for all $\vx' \in \R^{n-1}$ and $d \in \R_+$. We conclude that the integrals in~\eqref{eqn:gauge-fixed-angle-proof-1} vanish separately i.e.
\begin{equation}
\label{eqn:gauge-fixed-angle-proog-2}
\int_0^\infty
f((\vx',0)+t\vuu(\theta_0))
\,\D t
=0
\end{equation}
and
\begin{equation}
\label{eqn:gauge-fixed-angle-proof-3}
\int_0^\infty
f((\vx',0)+t\vv(\theta_0))
\,\D t
=0
\end{equation}
for all $\vx' \in \R^{n-1}$. Then we define the function
\begin{equation}
\psi(\vx',x_n)
\coloneqq
-\int_0^\infty
f((\vx,x_n)+t\vuu(\theta_0))
\,\D t.
\end{equation}
It follows from~\eqref{eqn:gauge-fixed-angle-proog-2} that $\psi(\vx',0) = 0$ for all $\vx' \in \R^{n-1}$. Since $f$ is smooth and compactly supported, the same holds for $\psi$. Also, it is straightforward to check that $\nabla_{\vuu(\theta_0)} \psi(\vx',x_n) = f(\vx',x_n)$ for all $(\vx',x_n) \in \R^n_+$.

As the final step we construct the function
\begin{equation}
\varphi(\vx',x_n)
=
\int_0^\infty
\psi((\vx',x_n)+t\vv(\theta_0))
\,\D t.
\end{equation}
We see that $\varphi(\vx,0) = 0$ for all $\vx' \in \R^{n-1}$. In fact, it follows from~\eqref{eqn:gauge-fixed-angle-proof-3} that
\begin{equation}
\int_0^\infty
\nabla_{\vuu(\theta_0)}
\psi((\vx',0)+t\vv(\theta_0))
\,\D t
=
\int_0^\infty
f((\vx',0)+t\vv(\theta_0))
\,\D t
=
0
\end{equation}
and therefore also
\begin{equation}
\varphi(\vx',0)
=
\int_0^\infty
\psi((\vx',0)+t\vv(\theta_0))
\,\D t
=
0.
\end{equation}
We have $\varphi \in C^\infty_c(\R^n_+)$ since $\psi \in C^\infty_c(\R^n_+)$ and by a straightforward computation we see that 
\begin{equation}
\nabla_{\vuu(\theta_0)}
\nabla_{\vv(\theta_0)}
\varphi(\vx',x_n)
=
f(\vx',x_n)
\end{equation}
for all $(\vx',x_n) \in \R^n_+$.
\end{proof}

\subsection{The full data transform}

Assume that $n \geq 3$. In this section, the gliding surface is assumed to be $\R^{n-1} \times \{0\}$ and we only consider constant vectors $\vuu$ and $\vv$. Our observation is that, even full data i.e. the knowledge of $\mathcal{R}f(\vx',d,\theta)$ for all $\vx' \in \R^{n-1}$, $d \in \R_+$ and $\theta \in \sphere^{n-2}$ is not enough to determine $f$ uniquely.
In this case, the head-wave transform of a sufficiently smooth function $f$ is
\begin{equation}
\mathcal{R}f(\vx',d,\theta)
=
\int_0^\infty
f((\vx',0)+t\vuu)
\,\D t
+
\int_0^d
f(\vx'+t\theta,0)
\,\D t
+
\int_0^\infty
f((\vx'+d\theta,0)+t\vv)
\,\D t.
\end{equation}

\begin{example}
\label{ex:general-non-uniqueness}
Consider a function $f \in C^\infty_c(\R^n)$ and let $h \colon \R \to \R$ be a non-zero smooth function so that $h(0) = 0$, $h(x) \to 0$ as $x \to \pm\infty$ and $h'(0) = 0$. Define the function $g \colon \R^n \to \R$ by $g(\vx',x_n) \coloneqq f(\vx',x_n) + h'(x_n)$. Then the head-wave transforms of the functions $g$ and $f$ are equal. Indeed, for all $\vx' \in \R^{n-1}$, $d \in \R_+$ and $\theta \in \sphere^{n-2}$ we have
\begin{equation}
\mathcal{R}g(\vx',d,\theta)
=
\mathcal{R}f(\vx',d,\theta)
+
\int_0^\infty
h'(tu_n)
\,\D t
+
\int_0^d
h'(0)
\,\D t
+
\int_0^\infty
h'(tv_n)
\,\D t
=
\mathcal{R}f(\vx',d,\theta).
\end{equation}
\end{example}

\section{The head-wave transform with gliding on a general submanifold}
\label{sec:gliding-on-general-surface}

In this section we generalize the head-wave transform to situations where the gliding surface is not necessarily the flat hypersurface $\{x_n = 0\} \subseteq \R^n$ but a more general submanifold. We extend our results for the flat case into this new general case in $2$-dimensions. It is clear that to treat the most general settings allowed by the general definition, one needs more sophisticated tools. Such study is outside the scope of this article.

\subsection{Generalized head-wave transform}

Let $M$ be an embedded submanifold of $\R^n$. Points on $\R^n$ are denoted by $\vx$ and points on $M$ by $\vx'$. We equip $M$ with the Riemannian metric induced on it by the ambient Euclidean metric of $\R^n$. Assume that $M$ is complete and orientable.
Let $SM$ be the unit sphere bundle of $M$ by $(\vx',\theta) \in SM$ where $\vx' \in M$ and $\theta \in S_{\vx'}M$. The fiber of the bundle at $x \in M$ is the unit sphere in $T_{\vx'}M$ i.e. $\theta \in S_{\vx'}M$ if $\abs{\theta} = 1$ where the length of $\vv$ is defined using the Riemannian metric of $M$.
Let $\vuu$ and $\vv$ be smooth functions $SM \to \R^n$ and suppose that $\abs{\vuu(\vx',\theta)} = \abs{\vv(\vx',\theta)} = 1$ for all $(\vx',\theta) \in SM$.

We impose the following restrictions on $\vuu$ and $\vv$. Let $N$ be a unit normal vector field of $M$. Then we can decompose the vectors $\vuu$ and $\vv$ as
\begin{equation}
\vuu(\vx',\theta)
=
\vuu_N(\vx',\theta)N(\vx')
+
\vuu_M(\vx',\theta)
\end{equation}
and
\begin{equation}
\vv(\vx',\theta)
=
\vv_N(\vx',\theta)N(\vx')
+
\vv_M(\vx',\theta)
\end{equation}
where $\vuu_N$ and $\vv_N$ are smooth scalar functions on $SM$ and $\vuu_M$ and $\vv_M$ are smooth sections of the pullback bundle $\pi^*TM$ over $SM$ where $\pi \colon SM \to M$ is the bundle map. As such $\vuu_M(\vx',\theta),\vv_M(\vx',\theta) \in T_{\vx'}M$.
We assume that
\begin{enumerate}
    \item[(i)] $\vuu_M(\vx',\theta) \parallel \theta$ and $\vv_M(\vx',\theta) \parallel \theta$ for all $(\vx',\theta) \in SM$, and
    
    \item[(ii)] $\vuu_N(\vx,\theta) = \vuu(\vx',\theta) \cdot N(\vx') > 0$ and $\vv(\vx',\theta) \cdot N(\vx') > 0$ for all $(\vx',\theta) \in SM$.
\end{enumerate}

For a smooth compactly supported function $f \in C^\infty_c(\R^n)$ define its head-wave transform $\mathcal{R}f \colon SM \times \R_+ \to \R$ by
\begin{equation}
\label{eqn:general-hwt}
\mathcal{R}f(\vx',\theta,d)
\coloneqq
\int_0^\infty
f(\vx' + s\vuu(\vx',\theta))
\,ds
+
\int_0^d
f(\gamma_{\vx',\theta}(s))
\,ds
+
\int_0^\infty
f(\gamma_{\vx',\theta}(d) + s\vv(\gamma_{\vx',\theta}(d),\dot\gamma_{\vx',\theta}(d)))
\,ds
\end{equation}
where $\gamma_{\vx',\theta}$ is the unique unit speed geodesic of $M$ with $\gamma_{\vx',\theta}(0) = \vx'$ and $\dot\gamma_{\vx',\theta}(0) = \theta$. 
The additions in~\eqref{eqn:general-hwt} are well defined since $M$ is an embedded submanifold and as such is identified with a subset of $\R^n$. For the same reason, it makes sense to evaluate $f$ for example at $\vx' \in M$.

Let $M = \R^{n-1} \approx \R^{n-1} \times \{0\} \subset \R^n$ in which case $SM = \R^{n-1} \times \sphere^{n-2}$. We freely identify the points $\vx' \in M$ with points $(\vx',0) \in \R^n$. Then the head-wave transform reduces to
\begin{equation}
\mathcal{R}f(\vx',\theta,d)
=
\int_0^\infty
f(\vx'+s\vuu(\vx',\theta))
\,ds
+
\int_0^d
f(\vx' + s\theta)
\,ds
+
\int_0^\infty
f(\vx'+d\theta+s\vv(\vx'+d\theta,\theta))
\,ds.
\end{equation}

This is the same transform we defined in Section~\ref{sec:introduction} and studied in Sections~\ref{sec:refraction-plane} and~\ref{sec:general-dimensions}.

\subsection{Gliding on smooth curves}

In this section we consider the head-wave transform as introduced in the previous section, but in $\R^2$. Since the gliding surface in this case is just a curve, the only geodesics are line segments of this curve. Hence the definitions reduce to the following.

Let $\gamma \colon \R \to \R^2$ be a smooth curve without self-intersections. Let $\vuu$ and $\vv$ be smooth vector fields along $\gamma$ i.e. for all $t \in \R$ it holds that $\vuu(t),\vv(t) \in T_{\gamma(t)}\R^2 \approx \R^2$. Assume that $\vuu$ and $\vv$ have unit length and that
\begin{equation}
\label{eqn:assumptions-on-u-v-curve-case}
\vuu(t) \cdot \dot\gamma(t) < 0,
\quad
\vuu(t) \cdot \dot\gamma^\perp(t) > 0,
\quad
\vv(t) \cdot \dot\gamma(t) > 0
\quad\text{and}\quad
\vv(t) \cdot \dot\gamma^\perp(t) > 0
\end{equation}
where $\dot\gamma^\perp(t)$ is the rotation of $\dot\gamma(t)$ by $90$ degrees counter clockwise. Let $f \in C^\infty_c(\R^2)$ be a smooth and compactly supported function. The $\gamma$-head-wave transform of $f$ is the function $\mathcal{R}f \colon \R \times \R_+ \to \R$ defined by the formula
\begin{equation}
\mathcal{R}f(t_0,d)
=
\int_0^\infty
f(\gamma(t_0) + s\vuu(t_0))
\,ds
+
\int_0^d
f(\gamma(t_0+s))
\,ds
+
\int_0^\infty
f(\gamma(t_0+d) + s\vv(t_0+d))
\,ds
\end{equation}

To analyze the $\gamma$-head-wave transform we have to make some assumptions on the geometry of $\gamma$. We assume that the vector fields $\dot\gamma$ and $\dot\gamma^\perp$ along $\gamma$ are extendable in the sense that there is a neighborhood $U \subseteq \R^2$ of $\gamma(\R)$ and smooth vector fields $\ve \colon U \to \R^2$ and $\ve^\perp \colon U \to \R^2$ so that $\ve(\gamma(t)) = \dot\gamma(t)$ and $\ve^\perp(\gamma(t)) = \dot\gamma^\perp(t)$ for all $t \in \R$. We assume in addition that $\{\ve,\ve^\perp\}$ is a global frame over $U$ i.e. the vectors $\ve(p)$ and $\ve^\perp(p)$ form a basis of $T_pU \approx \R^2$ for all $p \in U$. From now on we freely identify any tangent vectors with points in $\R^2$ without an explicit mention.

Consider a function $f \in C^\infty_c(U)$ where $U$ is as above. Suppose there is a function $\tilde f \colon \R \to \R$ so that
\begin{equation}
f(p)
=
f((p \cdot \ve)\ve + (p \cdot \ve^\perp)\ve^\perp)
=
\tilde f(p \cdot \ve)
\end{equation}
for all $p \in U$. Let us denote
\begin{equation}
\gamma(t)
=
(\gamma(t) \cdot \dot\gamma(t))\dot\gamma(t)
+
(\gamma(t) \cdot \dot\gamma^\perp(t))\dot\gamma^\perp(t)
\eqqcolon
\gamma_1(t)\dot\gamma(t)
+
\gamma_2(t)\dot\gamma^\perp(t)
\end{equation}
and similarly for the vector fields $\vuu$ and $\vv$. In particular, $u_1(t) = \vuu(t) \cdot \dot\gamma(t)$ and $v_1(t) = \vv(t) \cdot \dot\gamma(t)$. Given these assumption our $\gamma$-head-wave transform can be written as
\begin{equation}
\label{eqn:gamma-transform-reduction-1}
\mathcal{R}f(t_0,d)
=
\int_0^\infty
\tilde f(\gamma_1(t_0) + su_1(t_0))
\,ds
+
\int_0^d
\tilde f(\gamma_1(t_0+s))
\,ds
+
\int_0^\infty
\tilde f(\gamma_1(t_0+d) + sv_1(t_0+d))
\,ds
\end{equation}
It follows from our assumptions~\eqref{eqn:assumptions-on-u-v-curve-case} that after performing appropriate changes of coordinates, equation~\eqref{eqn:gamma-transform-reduction-1} reduces to
\begin{equation}
\label{eqn:gamma-transform-reduction-2}
\mathcal{R}f(t_0,d)
=
-
\frac{1}{u_1(t)}
\int_{-\infty}^{\gamma_1(t_0)}
\tilde f(s)
\,ds
+
\int_{t_0}^{d+t_0}
\tilde f(\gamma_1(s))
\,ds
+
\frac{1}{v_1(t_0+d)}
\int_{\gamma_1(t_0+d)}^\infty
\tilde f(s)
\,ds.
\end{equation}
This is due to the fact that $u_1(t) = \vuu(t) \cdot \dot\gamma(t) < 0$ and $v_1(t) = \vv(t) \cdot \dot\gamma(t) > 0$.
Given this form for the transform we are able to prove the following.

\begin{theorem}[Inversion formula for gliding on a smooth curve]
\label{thm:inversion-smooth-curve}
Let $\gamma \colon \R \to \R^2$ be a smooth curve without self-intersetions. Let $\vuu$ and $\vv$ be smooth unit vector fields along $\gamma$ with
\begin{equation}
\vuu(t) \cdot \dot\gamma(t) < 0,
\quad
\vuu(t) \cdot \dot\gamma^\perp(t) > 0,
\quad
\vv(t) \cdot \dot\gamma(t) > 0
\quad\text{and}\quad
\vv(t) \cdot \dot\gamma^\perp(t) > 0.
\end{equation}
Suppose that $\vuu$ and $\vv$ are extendable to a neighborhood $U$ of $\gamma(\R)$ in the sense explained above.
Let $f \in C^\infty_c(U)$ be a smooth compactly supported function and suppose there is a function $\tilde f \colon \R \to \R$ satisfying $f(p) = \tilde f(p \cdot \ve)$ where $\ve$ is the extension of $\dot\gamma$ defined as above.

\noindent Furthermore, assume that
\begin{equation}
-
\left(
\frac{u_1'(t_0)}{u_1(t_0)^2}
+
\frac{v_1'(t_0)}{v_1(t_0)^2}
\right)
\left(
1
-
\frac{\gamma_1'(t_0)}{v_1(t_0)}
\right)
-
\frac{v_1'(t_0)}{v_1(t_0)^2}
\left(
\frac{1}{u_1(t_0)}
+
\frac{1}{v_1(t_0)}
\right)
\neq
0
\end{equation}
for all $t_0 \in \R$. Then $f$ can be uniquely and explicitly be recovered from the knowledge of its $\gamma$-head-wave transform $\mathcal{R}f(t_0,d)$ and the integral
\begin{equation}
\int_{-\infty}^\infty
\tilde f(s)
\,ds.
\end{equation}
\end{theorem}

\begin{proof}
Under the assumptions stated in the theorem, the $\gamma$-head-wave transform can be brought to the form~\eqref{eqn:gamma-transform-reduction-2}. We take the derivatives of this equation with respect to $t_0$ and $d$ to find that 
\begin{equation}
\partial_{t_0}\mathcal{R}f(t_0,0)
=
\frac{u_1'(t_0)}{u_1(t_0)^2}
\int_{-\infty}^\infty
\tilde f(s)
\,ds
-
\left(
\frac{u_1'(t_0)}{u_1(t_0)^2}
+
\frac{v_1'(t_0)}{v_1(t_0)^2}
\right)
g(\gamma_1(t_0))
-
\left(
\frac{1}{u_1(t_0)}
+
\frac{1}{v_1(t_0)}
\right)
\tilde f(\gamma_1(t_0))
\end{equation}
and
\begin{equation}
\left.
\partial_d\mathcal{R}f(t_0,d)
\right|_{d=0}
=
\left(
1
-
\frac{\gamma'_1(t_0)}{v_1(t_0)}
\right)
\tilde f(\gamma_1(t_0))
-
\frac{v_1'(t_0)}{v_1(t_0)^2}
g(\gamma_1(t_0))
\end{equation}
where
\begin{equation}
g(x)
=
\int_x^\infty
\tilde f(s)
\,ds.
\end{equation}
To simplify these equations we let
\begin{equation}
\zeta(x)
\coloneqq
\frac{u_1'(x)}{u_1(x)^2}
\int_{-\infty}^\infty
\tilde f(s)
\,ds,
\quad
\alpha(x)
=
\frac{1}{u_1(x)}
+
\frac{1}{v_1(x)}
\quad\text{and}\quad
\beta(x)
=
\frac{1}{v_1(x)}.
\end{equation}
Then it holds that
\begin{equation}
\partial_{t_0}\mathcal{R}f(t_0,0)
=
\zeta(t_0)
-
\alpha(t_0)
\tilde f(\gamma_1(t_0))
+
\alpha'(t_0)
g(\gamma_1(t_0))
\end{equation}
and
\begin{equation}
\left.
\partial_d\mathcal{R}f(t_0,d)
\right|_{d=0}
=
(1 - \gamma_1'(t_0)\beta(t_0))\tilde f(\gamma_1(t_0))
+
\beta'(t_0)g(\gamma_1(t_0)).
\end{equation}
Hence
\begin{equation}
\label{eqn:proof-of-inversion-formula-curved}
\alpha'(t_0)\partial_d\mathcal{R}f(t_0,0)
-
\beta'(t_0)[\left.\partial_{t_0}\mathcal{R}f(t_0,d)\right|_{d=0}
-
\zeta(t_0)]
=
[
\alpha'(t_0)(1-\gamma_1'(t_0)\beta(t_0))
+
\beta'(t_0)\alpha(t_0)
]
\tilde f(\gamma_1(t_0)).
\end{equation}
Therefore we are able to recover $\tilde f$ and hence $f$ in the neighborhood $U$ of $\gamma$ given that the coefficient on the right-hand side in~\eqref{eqn:proof-of-inversion-formula-curved} is non-zero which ammounts to
\begin{equation}
-
\left(
\frac{u_1'(t_0)}{u_1(t_0)^2}
+
\frac{v_1'(t_0)}{v_1(t_0)^2}
\right)
\left(
1
-
\frac{\gamma_1'(t_0)}{v_1(t_0)}
\right)
-
\frac{v_1'(t_0)}{v_1(t_0)^2}
\left(
\frac{1}{u_1(t_0)}
+
\frac{1}{v_1(t_0)}
\right)
\neq
0
\end{equation}
for all $t_0 \in \R$ exactly as we assumed.
\end{proof}

To state our next theorem we extend the vector fields $\vuu$ and $\vv$ along $\gamma$ to vector fields in $\R^2$ so that their integral curves are straight lines. The extensions are still denoted by $\vuu$ and $\vv$. Using the extensions we define the variable coefficient partial differential operators
\begin{equation}
P_\vuu(p) \coloneqq \vuu(p) \cdot \nabla
\quad\text{and}\quad
P_\vv(p) \coloneqq \vv(p) \cdot \nabla.
\end{equation}
Let us also define the matrix
\begin{equation}
A_{\vuu,\vv}(p)
\coloneqq
\begin{pmatrix}
\vuu(p) & \vv(p)
\end{pmatrix}
=
\begin{pmatrix}
u_1(p) & v_1(p) \\
u_2(p) & v_2(p)
\end{pmatrix}.
\end{equation}
Theorem~\ref{thm:gauge-curved-2d} requires the assumption that $f \in C^\infty_c(\R^2)$ satisfies 
\begin{equation}
\label{eqn:extra-assumption-curved}
\diver(\vuu(p))
\int_0^\infty
f(p + s\vuu(p))
\,ds
-
\diver(\vv(p))
\int_0^\infty
f(p + s\vv(p))
\,ds
=
0
\quad
\text{for all}
\quad
p \in \R^2.
\end{equation}
However, as opposed to Theorem~\ref{thm:inversion-smooth-curve} we do not need as restrictive assumptions on the curve $\gamma$. In fact, the restrictions imposed on $\gamma$ are quite implicit we are assuming that the vector fields $\vuu$ and $\vv$ can be extended with integral curves as straight lines. However, at the same time we are assuming that equations~\eqref{eqn:assumptions-on-u-v-curve-case} are valid.

\begin{theorem}[Gauge description for gliding on a smooth curve]
\label{thm:gauge-curved-2d}
Assume that the vector fields are smooth and always linearly independent. Suppose that $\vuu$ and $\vv$ can be extended to $\R^2$ smoothly so that their integral curves are straight lines. Let $f \in C^\infty_c(\R^2)$ be a function satisfying~\eqref{eqn:extra-assumption-curved} and assume that $f(\gamma(t)) = 0$ for all $t \in \R$. Then the following are equivalent:
\begin{enumerate}
    \item[(i)] The $\gamma$-head-wave transform of $f$ vanishes i.e. $\mathcal{R}f(t_0,d) = 0$ for all $t_0 \in \R$ and $d \in \R_+$.

    \item[(ii)] There exists a function $\varphi \in C^\infty_c(\R^2)$ vanishing along $\gamma$ so that
    \begin{equation}
    P_\vv(p)(\det(A_{\vuu,\vv}(p))^{-1}P_\vuu(p)\varphi(p))
    =
    f(p)
    \end{equation}
    for all $p \in \R^2$ and also
    \begin{equation}
    P_\vuu(p)(\det(A_{\vuu,\vv}(p))^{-1}P_\vv(p)\varphi(p))
    =
    f(p)
    \end{equation}
    for all $p \in \R^2$.
\end{enumerate}
\end{theorem}

\begin{proof}
Assume first that (ii) holds. Then there are functions $\psi_1(p)$ and $\psi_2(p)$ so that
\begin{equation}
f(p) = P_{\vv}(p)\psi_1(p)
\quad\text{and}\quad
f(p) = P_{\vuu}(p)\psi_2(p).
\end{equation}
Using the fact that the integral curves of $\vv$ and $\vuu$ are straight lines, it follows that
\begin{equation}
\mathcal{R}f(t_0,d)
=
\int_0^\infty
\vuu(t_0)\cdot\nabla\psi_2(\gamma(t_0) + s\vuu(t_0))
\,\D s
+
\int_0^\infty
\vv(t_0+d)\cdot\nabla\psi_1(\gamma(t_0+d) + s\vv(t_0+d))
\,\D s.
\end{equation}
This expression equals $0$ due to the fundamental theorem of calculus. Hence (i) is satisfied.

To prove the other direction suppose that $\mathcal{R}f(t_0,d) = 0$ for all $t_0 \in \R$ and $d \in \R_+$. Since $f(\gamma(t)) = 0$ for all $t \in \R$ this vanishing implies that
\begin{equation}
\label{eqn:sperate-vanishing-curved}
\int_0^\infty
f(\gamma(t_0) + s\vuu(t_0))
\,ds
=
0
\quad
\text{and}
\quad
\int_0^\infty
f(\gamma(t_0) + s\vv(t_0))
\,ds
=
0
\end{equation}
for all $t_0 \in \R$. We define the auxiliary functions
\begin{equation}
\psi_\vuu(p)
\coloneqq
\int_0^\infty
f(p + s\vuu(p))
\,ds
\quad\text{and}\quad
\psi_\vv(p)
\coloneqq
\int_0^\infty
f(p + s\vv(p))
\,ds
\end{equation}
for all $p \in \R^2$. Since $f$ is smooth and compactly supported, so are $\psi_\vuu$ and $\psi_\vv$. In addition, $\psi_\vuu(\gamma(t)) = 0$ and $\psi_\vv(\gamma(t)) = 0$ for all $t \in \R$ due to equation~\eqref{eqn:sperate-vanishing-curved}. A simple calculation verifies that $P_\vuu(p)\psi_\vuu(p) = f(p)$ and $P_\vv(p)\psi_\vv(p) = f(p)$.

We use these auxiliary functions to construct the smooth and compactly supported $1$-form $\omega \coloneqq \omega_1dx + \omega_2dy$ where
\begin{equation}
\omega_1(p)
\coloneqq
v_2(p)\psi_\vv(p)
-
u_2(p)\psi_\vuu(p)
\quad
\text{and}
\quad
\omega_2(p)
\coloneqq
u_1(p)\psi_\vuu(p)
-
v_1(p)\psi_\vv(p)
\end{equation}
for all $p \in \R^2$. Then exactly as in the proof of Theorem~\ref{thm:gauge-2d-var-angles} we can prove that $\omega$ is closed, due to assumption~\eqref{eqn:extra-assumption-curved} on the function $f$. Thus $\omega$ is exact i.e. there is a smooth and compactly supported function $\varphi \in C^\infty_c(\R^2)$ so that $d\varphi = \omega$. It follows that
\begin{equation}
P_\vuu(p)\varphi
=
u_1(p)\omega_1(p) + u_2(p)\omega_2(p)
=
\det(A_{\vuu,\vv}(p))\psi_\vv
\end{equation}
and thus
\begin{equation}
P_\vv(p)(\det(A_{\vuu,\vv}(p))^{-1}P_\vuu(p)\varphi(p))
=
f(p)
\end{equation}
for all $p \in \R^2$. A similar computation shows that
\begin{equation}
P_\vuu(p)(\det(A_{\vuu,\vv}(p))^{-1}P_\vv(p)\varphi(p))
=
f(p)
\end{equation}
for all $p \in \R^2$ proving that $\varphi$ is the function we set out to construct.
\end{proof}

\section*{Acknowledgements}

MVdH was supported by the National Science Foundation under grant DMS-2407456, the Simons Foundation under the MATH + X program and the corporate members of the Geo-Mathematical Imaging Group at Rice University. We thank the anonymous referees for the valuable remarks and suggestions.





\bibliographystyle{abbrv}
\bibliography{references}

\end{document}